\documentclass[final,reqno]{siamltex}
\usepackage{latexsym,amsmath,amssymb,amsfonts,mathrsfs}
\usepackage{epsf,graphicx,epsfig,color,cite,cases}
\usepackage{subfigure,graphics,multirow,marginnote,enumerate,bm}
\newcommand{\Om}{\Omega}

\newcommand{\pa}{\partial}

\newcommand{\ov}{\overline}
\newcommand{\I}{{\rm Im}}
\newcommand{\Rt}{{\rm Re}}

\newcommand{\wid}{\widetilde}

\newcommand{\na}{\nabla}
\newcommand{\mat}{\mathbb}

\newcommand{\R}{{\mat R}}
\newcommand{\Z}{{\mat Z}}
\newcommand{\N}{{\mat N}}

\newcommand{\Sp}{{\mat S}}

\newcommand{\be}{\begin{eqnarray}}
\newcommand{\ben}{\begin{eqnarray*}}
\newcommand{\en}{\end{eqnarray}}
\newcommand{\enn}{\end{eqnarray*}}

\newtheorem{remark}[theorem]{Remark}

\begin{document}
\renewcommand{\theequation}{\arabic{section}.\arabic{equation}}

\title{\bf The obstalce scattering for the biharmonic equation}
\author{Chengyu Wu\thanks{School of Mathematics and Statistics, Xi'an Jiaotong University,
Xi'an 710049, Shaanxi, China ({\tt wucy99@stu.xjtu.edu.cn})}
\and
Jiaqing Yang\thanks{School of Mathematics and Statistics, Xi'an Jiaotong University,
Xi'an 710049, Shaanxi, China ({\tt jiaq.yang@xjtu.edu.cn})}
}
\date{}
\maketitle


\begin{abstract}
  In this paper, we consider the obstacle scattering problem for biharmonic equations with a Dirichlet boundary condition in both two and three dimensions. Some basic properties are first derived for the biharmonic scattering solutions, which leads to a simple criterion for the uniqueness of the direct problem. Then a new type far-field pattern is introduced, where the correspondence between the far-field pattern and scattered field is established. Based on these properties, we prove the well-posedness of the direct problem in associated function spaces by utilizing the boundary integral equation method, which relys on a natural decomposition of the biharmonic operator and the theory of the pseudodifferential operator. Furthermore, the inverse problem for determining the obstacle is studied. By establishing some novel reciprocity relations between the far-field pattern and scattered field, we show that the obstacle can be uniquely recovered from the measurements at a fixed frequency. 
\end{abstract}



\pagestyle{myheadings}
\thispagestyle{plain}
\markboth{C. Wu and J. Yang}{Direct and inverse biharmonic obstacle scattering problem}

\section{Introduction}\label{sec1}
\setcounter{equation}{0}
The biharmonic scattering problems have important applications in various scientific fields and thus have arisen more interests in recent years. They play a significant part in the study of elasticity and the theory of vibration of beams, such as the beam equation \cite{FHG10}, the hinged plate configurations \cite{FHG10}, and the Stokes equation \cite{GW08}, and the scattering by grating stacks \cite{NR09}. 

The present paper concerns the direct and inverse biharmonic scattering problems by impenetrable obstacles. Denote by $\Om\in C^{3,\alpha}$ a bounded domain in $\R^d$ $(d=2,3)$ with connected complement. Consider the following biharmonic obstacle scattering problem 
\be\label{1.1}
\left\{
\begin{array}{ll}
	\Delta^2 u-k^4u=0~~~&{\rm in}~\R^d\setminus\ov\Om,\\ [2mm]
	\mathcal{B}(u)=(0,0)~~~&{\rm on}~\pa\Om, 
\end{array}
\right.
\en
where $k>0$ is the wave number, $u=u^i+u^s$ denotes the total field in $\R^d\setminus\ov\Om$ with $u^i$ the incident wave and $u^s$ the scattered wave, and $\mathcal{B}$ stands for the boundary condition on $\pa\Om$. There are various boundary conditions for the biharmonic scattering (cf. \cite{LC20,FHG10,GW08}), such as the Dirichlet condition $\mathcal{B}_D(u)=(u,\pa_nu)$ with $n$ the unit exterior normal on $\pa\Om$, the Navier condition $\mathcal{B}_N(u)=(u,\Delta u)$, the Neumann condition $\mathcal{B}(u)=(\Delta u,\pa_n\Delta u)$ (this is actually a special case of the Neumann condition) and $\mathcal{B}(u)=(u,\pa_n\Delta u),(\pa_nu,\Delta u),(\pa_nu,\pa_n\Delta u)$. In this paper, we mainly consider the Dirichlet condition $\mathcal{B}=\mathcal{B}_D$. To ensure the well-posedness of problem (\ref{1.1}), we impose an analogue of the classical Sommerfeld radiation condition (cf. \cite{TV18,HP23,PXY21}), i.e., 
\be\label{1.2}
  \displaystyle\pa_rw-ikw=o\left(r^{-\frac{d-1}{2}}\right),~~r=|x|\rightarrow\infty,~w=u^s,\Delta u^s, 
\en
uniformly in all directions $\hat x=x/|x|$. In this work, we will define a new type far-field pattern for the biharmonic scattering and study the inverse problem of uniquely determining the obstacle $\Om$ from the knowledge of the scattered fields or far-field patterns at a fixed frequency. 

Comparing to the scattering problems in the acoustic, elastic and electromagnetic cases, the biharmonic scattering problems are much less studied till now. The high order of the differential operator bring numerous difficulties in the reserach and many classical methods do not work any longer. For the direct biharmonic obstacle scattering problems, we refer to \cite{LC20} where the Dirichlet-to-Neumann map for biharmonic scattering was defined in two dimensions and the corresponding Fredholm properties were derived, which yielded an equivalent variational formulation and the well-posedness of biharmonic scattering with Dirichlet condition and other boundary conditions followed. Besides the variational method, the boundary integral equation method is also an usual method to derive the well-posedness in the scattering theory. For the biharmonic scattering, the boundary integral equation method was considered in \cite{HP23}, where the well-posedness of the Dirichlet problem was obtained and relating convergence analysis was carried out. However, the results in \cite{HP23} was only obtained in two dimensions with the boundary $\pa\Om$ required to be analytic and parametrizable, which is a rather strong restriction. As for the inverse scattering by biharmonic obstacles, in \cite{LA20}, the unique recovery of a Dirichlet obstacle was obtained by measuring the scattered field with the incident point source at a circle. Moreover, the linear sampling method was then justified. Here we further refer some papers concerning other types problems in the biharmonic scattering. In \cite{TV18}, the Saito's formula in the biharmonic case was proved, which leads to the unique identification of the perturbation of the biharmonic operator. The biharmonic inverse source problems are considered in \cite{PX21,PX22,PXY21}. For the biharmonic scattering in the nonlinear case, we refer the readers to \cite{MJV22} and the references quoted there. We also mention some work about the inverse boundary value problems for the bi- and polyharmonic operators, see e.g., \cite{YA16,AH17,MKG12,YY14,KMG14}, where the well-known complex geometrical optics solutions are generalized to the biharmonic case. 

Inspired by the natural decomposition of the biharmonic operator into the Helmholtz and modified Helmholtz operators, we investigate on the direct and inverse biharmonic obstacle scattering problems. First, we study the basic properties of the biharmonic scattering solutions, which leads to a simple criterion for the uniqueness of the direct problem. Further, we define a new type far-field pattern for the biharmonic scattering and establish its correspondence to the scattered field. For the well-posedness of the direct problem, in fact, it is easy to find that the biharmonic equation with some boundary conditions, such as the Navier condition, after decomposition can be solved by the usual methods in the acoustic scattering. However, different from the acoustic case, it is hard to develop a unified way to deal with all the boundary conditions simultaneously, since in the Dirichlet case the biharmonic equation is spiltted to an interior transmission problem in the exterior domain, which is known that can not be managed by the classical methods. With the variational method being studied in great details in \cite{LC20}, the corresponding boundary integral equation method, which has its own significance in both mathematics and practice, is much less investigated, since the only relating work \cite{HP23} contains certain strong restrictions. We therefore further develop the boundary integral equation method for the Dirichlet case in the present study. After spiltting the biharmonic scattering solution into Helmholtz and modified Helmholtz parts, and expressing both parts as a combination of the single- and double-layer potentials, we reduce the direct problem to an equivalent boundary integral equation. Through an elaborate analysis, the Fredholm properties of the corresponding integral operators are derived utilizing the theory of the pseudodifferential operators and the well-posedness follows. In contrast to \cite{LC20,HP23}, here we obtain the well-posedness in both two and three dimensions. Furthermore, we weaken the conditions in \cite{HP23} to only requiring $\pa\Om\in C^{3,\alpha}$. Finally, we prove some uniqueness theorems in determining the obstacle. Applying the reciprocity relations of the far-field pattern and scattered field, we show that the biharmonic obstacle can be uniquely identified by the measurements of the scattered fields or far-field patterns at a fixed frequency, while almost all the preceding results on the inverse biharmonic scattering problems are using multi-frequency measurements. 

The rest of this paper is organized as follows. In Section \ref{sec2}, we fix some notations and recall some estimates about the fundamental solution. In Section \ref{sec3}, we discover some foundational properties for the biharmonic scattering, which yields a simple criterion for the uniqueness of the direct problem. Moreover, a new type far-field pattern is defined. Section \ref{sec3.5} is about the establishment of the boundary integral equation method, and the well-posedness of the direct problem is obtained for Dirichlet boundary condition. Finally, in Section \ref{sec6}, we prove the uniqueness results for the inverse problem of determing the obstacle by measurements at a fixed frequency, which follows from the reciprocity relations for far-field pattern and scattered field.

\section{Preliminaries}\label{sec2}
\setcounter{equation}{0}
In this section, we introduce some notations and important asymptotic expansions for the fundamental solution used throughout the paper. 

Denote by $B_r(x)$ the open disk (ball) centered at $x\in\R^d$ with radius $r>0$. For disks (balls) centered at the origin, we abbreviate by $B_r$. Denote by $H_\nu^{(1)}$ and $K_\nu$ the Hankel function of first kind and the Macdonald's function of order $\nu$, respectively. It is well known that the fundamental solution for $\Delta^2-k^4$ in $\R^d$ is given by 
\ben
  G_k(|x|)=
  \left\{
  \begin{array}{ll}
  	\displaystyle\frac{i}{8k^2}\left(H_0^{(1)}(k|x|)+\frac{2i}{\pi}K_0(k|x|)\right),~~&d=2,\\ [3mm]
  	\displaystyle\frac{1}{8\pi k^2|x|}\left(e^{ik|x|}-e^{-k|x|}\right),~~&d=3. 
  \end{array}
  \right.
\enn

We here recall some properties of these functions (for details see \cite{NL72,GW94,WFR66}). It holds that 
\be\label{2.1}
  H_\nu^{(1)}(x)=
  \left\{
  \begin{array}{ll}
  	O(|x|^{-\nu}),~~&\nu>0,\\ [1mm]
  	O(|\ln(x)|),~~&\nu=0, 
  \end{array}
  \right.
\en
as $x\rightarrow0+$. The Macdonald's function $K_\nu$ has the same asymptotic behavior as $x\rightarrow0+$. Further, for $\nu\geq0$ 
\be\label{2.2}
  &&H_\nu^{(1)}(x)=\sqrt{\frac{2}{\pi x}}e^{i(x-\frac{1}{2}\nu\pi-\frac{1}{4}\pi)}+O\left(x^{-\frac{3}{2}}\right), \\ \label{2.3}
  &&K_\nu(x)=\sqrt{\frac{\pi}{2x}}e^{-x}+O\left(\frac{e^{-x}}{x^{\frac{3}{2}}}\right), 
\en
as $x\rightarrow+\infty$. We also remind that 
\ben
  H_{-m}^{(1)}=(-1)^mH_m^{(1)}~~~~{\rm and}~~~~K_{-m}=K_m 
\enn
for all $m\in\Z$. By the explicit expression of $G_k$ we have 
\be\label{2.4}
  G_k(|x|)=
  \left\{
  \begin{array}{ll}
  	O(1),~~~&{\text{as}}~|x|\rightarrow0,\\ [2mm]
  	\displaystyle O\left(|x|^{-\frac{d-1}{2}}\right),~&{\text{as}}~|x|\rightarrow+\infty. 
  \end{array}
  \right.
\en
Moreover, from the recurrence relations of Hankel and Macdonald's functions, it follows that 
\ben
  \na G_k(|x|)=
   \left\{
  \begin{array}{ll}
  	\displaystyle-\frac{ikx}{8k^2|x|}\left(H_1^{(1)}(k|x|)+\frac{2i}{\pi}K_1(k|x|)\right),~~&d=2,\\ [3mm]
  	\displaystyle \frac{x}{8\pi k^2|x|^3}\left((ik|x|-1)e^{ik|x|}+(k|x|+1)e^{-k|x|}\right),~~&d=3, 
  \end{array}
  \right.
\enn
and 
\ben
  \Delta G_k(|x|)=
  \left\{
  \begin{array}{ll}
  	\displaystyle-\frac{i}{8}\left(H_0^{(1)}(k|x|)-\frac{2i}{\pi}K_0(k|x|)\right),~~&d=2,\\ [3mm]
  	\displaystyle-\frac{1}{8\pi |x|}\left(e^{ik|x|}+e^{-k|x|}\right),~~&d=3, 
  \end{array}
  \right.
\enn
which indicates the asymptotic behaviors
\be\label{2.5}
\na G_k(|x|)=
\left\{
\begin{array}{ll}
	O\left(|x|^{2-d}\right),~~~&{\text{as}}~|x|\rightarrow0,\\ [2mm]
	\displaystyle O\left(|x|^{-\frac{d-1}{2}}\right),~&{\text{as}}~|x|\rightarrow+\infty, 
\end{array}
\right.
\en
and 
\be\label{2.6}
\Delta G_k(|x|)=
\left\{
\begin{array}{ll}
	O(|\ln(|x|)|),~~~&d=2,\\ [2mm]
	\displaystyle O\left(|x|^{-1}\right),~&d=3, 
\end{array}
\right.
\en
as $|x|\rightarrow0$ and 
\be\label{2.7}
  \Delta G_k(|x|)=O\left(|x|^{-\frac{d-1}{2}}\right),~~~{\text{as}}~|x|\rightarrow+\infty. 
\en
Also we see that the fundamental solution $G_k$ satisfies the radiation condition (\ref{1.2}). 

Denote by $Y_l^m$ the spherical harmonics of order $l$. Denote by $h_l^{(1)}$ and $k_l$ the spherical and modified spherical Hankel functions of order $l$, respectively. It is well known that 
\be\label{2.8}
  h_l^{(1)}(x)=\sqrt{\frac{\pi}{2x}}H_{l+\frac{1}{2}}^{(1)}(x)~~~{\rm and}~~~k_l(x)=\sqrt{\frac{\pi}{2x}}K_{l+\frac{1}{2}}(x). 
\en
We refer to \cite{DR13,NL72,JC01} for the more properties of these functions. 


\section{Some basic properties}\label{sec3}
\setcounter{equation}{0}
In this section, we prove some basic properties and define a new type far-field pattern for the biharmonic scattered solutions, which is fundamental in the theory of biharmonic scattering. Particularly, a simple criterion for the uniqueness of the direct problem is discovered and the correspondence between the far-field and scattered field is established. 
	  
\begin{lemma}\label{lem3.1}
	Suppose $u^s\in C^4(\R^d\setminus\ov\Om)\cap C^3(\R^d\setminus\Om)$ satisfies $\Delta^2u^s-k^4u^s=0$ in $\R^d\setminus\ov\Om$ and the radiation condition {\rm (\ref{1.2})}, then 
	\ben
	  \int_{\pa B_R}(|\Delta u^s|^2+|u^s|^2)ds=O(1),~~{\rm as}~R\rightarrow+\infty. 
	\enn
\end{lemma}
\begin{proof}
	We choose $R>0$ large enough such that $\ov\Om\subset B_R$. From the radiation conditon (\ref{1.2}), we have that 
	\be\label{3.1}\nonumber
	  0~&&=\lim\limits_{R\rightarrow\infty}\int_{\pa B_R}\left|\pa_rw-ikw\right|^2ds \\
	  &&=\lim\limits_{R\rightarrow\infty}\int_{\pa B_R}\left(|\pa_rw|^2+k^2|w|^2+2k\I(w\pa_r\ov w)\right)ds
	\en
	for $w=u^s,\Delta u^s$. Integration by parts over $B_R\setminus\ov\Om$ yields that 
	\ben
	 \I\int_{\pa B_R}w\pa_r\ov wds=\I\int_{\pa\Om}w\pa_n\ov wds+\I\int_{B_R\setminus\ov\Om}w\Delta\ov wdx, 
	\enn
	which implies 
	\ben
	  &&\I\int_{\pa B_R}\Delta u^s\pa_r\Delta\ov u^sds+k^4 \I\int_{\pa B_R}u^s\pa_r\ov u^sds \\
	  =\;&&\I\int_{\pa\Om}\Delta u^s\pa_n\Delta\ov u^sds+k^4 \I\int_{\pa\Om}u^s\pa_n\ov u^sds+\I\int_{B_R\setminus\ov\Om}\Delta u^s(\Delta^2\ov u^s-k^4\ov u^s)dx \\
	  =\;&&\I\int_{\pa\Om}(\Delta u^s\pa_n\Delta\ov u^s+k^4u^s\pa_n\ov u^s)ds. 
	\enn
	Hence, by (\ref{3.1}) we derive that 
	\be\label{3.2}\nonumber
	  &&\lim\limits_{R\rightarrow\infty}\int_{\pa B_R}[|\pa_r\Delta u^s|^2+k^2|\Delta u^s|^2+k^4(|\pa_r u^s|^2+k^2|u^s|^2)]ds \\
	  =\;&&-2k\I\int_{\pa\Om}(\Delta u^s\pa_n\Delta\ov u^s+k^4u^s\pa_n\ov u^s)ds, 
	\en
	and the conclusion follows. 
\end{proof}
\begin{theorem}\label{thm3.2}
	Under the assumptions in Lemma {\rm\ref{lem3.1}}, for $x\in\R^d\setminus\ov\Om$ we have 
	\ben
	  u^s(x)=-\int_{\pa\Om}&&\left(u^s(y)\pa_{n(y)}\Delta_yG_k(|x-y|)+\Delta u^s(y)\pa_{n(y)}G_k(|x-y|)\right. \\
	  &&~\left.-G_k(|x-y|)\pa_n\Delta u^s(y)-\Delta_yG_k(|x-y|)\pa_nu^s(y)\right)ds(y). 
	\enn
\end{theorem}
\begin{proof}
	For fixed $x\in\R^d\setminus\ov\Om$, we choose $R>0$ sufficiently large such that $\ov\Om\cup\{x\}\subset B_R$. Let $\varepsilon>0$ be small enough such that $\ov{B_\varepsilon(x)}\subset B_R\setminus\ov\Om$. Denote $S_{R,\varepsilon}:=(B_R\setminus\ov \Om)\setminus\ov{ B_\varepsilon(x)}$. Then we have 
	\ben
	  0~&&=\int_{S_{R,\varepsilon}}\left(u^s(y)(\Delta_y^2-k^4)G_k(|x-y|)-G_k(|x-y|)(\Delta^2-k^4)u^s(y)\right)dy \\
	  &&=\int_{\pa S_{R,\varepsilon}}\left(u^s(y)\pa_{n(y)}\Delta_yG_k(|x-y|)+\Delta u^s(y)\pa_{n(y)}G_k(|x-y|)\right. \\
	  &&~\qquad\qquad\left.-G_k(|x-y|)\pa_n\Delta u^s(y)-\Delta_yG_k(|x-y|)\pa_nu^s(y)\right)ds(y) \\
	  &&=\left(-\int_{\pa\Om}+\int_{\pa B_R}-\int_{\pa B_\varepsilon(x)}\right)\left(u^s(y)\pa_{n(y)}\Delta_yG_k(|x-y|)+\Delta u^s(y)\pa_{n(y)}G_k(|x-y|)\right. \\
	  &&~\qquad\qquad\left.-G_k(|x-y|)\pa_n\Delta u^s(y)-\Delta_yG_k(|x-y|)\pa_nu^s(y)\right)ds(y) \\
	  &&=I_1+I_2+I_3, 
	\enn
	We see that 
	\ben
	  I_2=\int_{\pa B_R}&&\left(u^s(y)(\pa_{n(y)}-ik)\Delta_yG_k(|x-y|)+\Delta u^s(y)(\pa_{n(y)}-ik)G_k(|x-y|)\right. \\
	  &&~\left.-G_k(|x-y|)(\pa_r-ik)\Delta u^s(y)-\Delta_yG_k(|x-y|)(\pa_r-ik)u^s(y)\right)ds(y). 
	\enn
	By Lemma \ref{lem3.1}, the radiation condition (\ref{1.2}) and the fact that $G_k(|x-y|),\Delta_y G_k(|x-y|)=O(R^{-(d-1)/2})$, we get $I_2\rightarrow0$ as $R\rightarrow+\infty$. From the asymptoic behavior (\ref{2.4})-(\ref{2.6}), we obtain that 
	\ben
	  \int_{\pa B_\varepsilon(x)}&&\left(\Delta u^s(y)\pa_{n(y)}G_k(|x-y|)-G_k(|x-y|)\pa_n\Delta u^s(y)\right. \\
	  &&~\left.-\Delta_yG_k(|x-y|)\pa_nu^s(y)\right)ds(y)\rightarrow0 
	\enn
	as $\varepsilon\rightarrow0$. Further, it is well known that on $\pa B_\varepsilon(x)$ (see, e.g., \cite{FD06,NL72,DR13})
	\ben
	 \displaystyle\pa_{n(y)}\Delta_yG_k(|x-y|)=\frac{\Gamma(\frac{d}{2})}{2\pi^{\frac{d}{2}}}\frac{1}{|x-y|^{d-1}}+O\left(\frac{1}{|x-y|^{d-2}}\right) 
	\enn
	with $\Gamma$ the Gamma function, which implies that 
	\ben
	  \int_{\pa B_\varepsilon(x)}u^s(y)\pa_{n(y)}\Delta_yG_k(|x-y|)ds(y)\rightarrow u^s(x),~~{\text{as}}~\varepsilon\rightarrow0. 
	\enn
	Therefore, $I_3\rightarrow -u^s(x)$ as $\varepsilon\rightarrow0$, which ends the proof. 
\end{proof} 

Arguing analogously as above, we can get the representation theorem in bounded domains. 
\begin{theorem}\label{thm3.3}
	Suppose $u\in C^4(\Om)\cap C^3(\ov\Om)$ and $\Delta^2u-k^4u=0$ in $\Om$, then for $x\in\Om$ we have 
	\ben
	  u(x)=\int_{\pa\Om}&&\left(u(y)\pa_{n(y)}\Delta_yG_k(|x-y|)+\Delta u(y)\pa_{n(y)}G_k(|x-y|)\right. \\
	  &&~\left.-G_k(|x-y|)\pa_n\Delta u(y)-\Delta_yG_k(|x-y|)\pa_nu(y)\right)ds(y). 
	\enn
\end{theorem}

The following result is a simple application of Rellich's Lemma.  
\begin{theorem}\label{thm3.4}
	Let $u\in C^4(\R^d\setminus\ov\Om)$ solves $\Delta^2u-k^4u=0$ in $\R^d\setminus\ov\Om$. If further 
	\ben
		\lim\limits_{R\rightarrow\infty}\int_{\pa B_R}(|\Delta u|^2+|u|^2)ds=0, 
	\enn
	then $\Delta u-k^2u=0$ in $\R^d\setminus\ov\Om$. 
\end{theorem}
\begin{proof}
	Clearly, $(\Delta-k^2)u\in C^2(\R^d\setminus\ov\Om)$ is a solution to the Helmholtz equation with  
	\ben
	\int_{\pa B_R}|(\Delta-k^2)u|^2ds~&&\leq\int_{\pa B_R}(|(\Delta-k^2)u|^2+|(\Delta+k^2)u|^2)ds \\
	&&=\int_{\pa B_R}2(|\Delta u|^2+k^4|u|^2)ds\rightarrow0,~~R\rightarrow+\infty. 
	\enn
	Thus by Rellich's Lemma \cite[Theorem 3.5]{FD06}, we have $\Delta u-k^2u=0$ in $\R^d\setminus\ov\Om$. 
\end{proof}

Next we are interesting in the behavior of the biharmonic solutions outside some large disk (ball), which are essential for the later establishment of the uniqueness result for the direct biharmonic scattering problem. 
\begin{lemma}\label{lem3.12}
	Let $R>0$ be such that $\ov\Om\subset B_R$. Suppose $v^s\in C^2(\R^d\setminus\ov\Om)$ satisfies $\Delta v^s-k^2v^s=0$ in $\R^d\setminus\ov\Om$ and the classical Sommerfeld radiation condition, i.e., 
	\be\label{3.6}
	\pa_rv^s-ikv^s=o\left(r^{-\frac{d-1}{2}}\right),~~r=|x|\rightarrow\infty. 
	\en
	Then for $x\in\R^d\setminus B_R$ we have that 
	\ben
	  &&\displaystyle v^s(r,\theta)=\sum\limits_{m\in\Z}a_mK_m(kr)e^{im\theta},~~~~{\rm if}~d=2, \\ 
	  &&\displaystyle v^s(r,\theta,\varphi)=\sum\limits_{l\in\N}\sum_{m=-l}^{l}a_l^mk_l(kr)Y_l^m(\theta,\varphi),~~~~{\rm if}~d=3, 
	\enn
	where $a_m,a_l^m$ are constants. 
\end{lemma}
\begin{proof}
	The proof is analogously to the Helmholtz equation case, we thus omit it here. 
\end{proof}

Combining Lemma \ref{lem3.12}, the series expansion for the scattering solution to the Helmholtz equation and the fact that $\Delta^2-k^4=(\Delta+k^2)(\Delta-k^2)$, we immediately obtain the expansion for the biharmonic scattering solutions. 
\begin{theorem}\label{thm3.5}
	Let $R>0$ be such that $\ov\Om\subset B_R$. Suppose $u^s\in C^4(\R^d\setminus\ov\Om)$ satisfying $\Delta^2u^s-k^4u^s=0$ in $\R^d\setminus\ov \Om$ and the radiation condition {\rm(\ref{1.2})}, then for $x\in\R^d\setminus B_R$, 
	\be\label{3.3}
	  &&u^s(r,\theta)=\sum\limits_{m\in\Z}\left(a_mH_m^{(1)}(kr)+b_mK_m(kr)\right)e^{im\theta},~~~~{\rm if}~d=2, \\ \label{3.35}
	  &&u^s(r,\theta,\varphi)=\sum\limits_{l\in\N}\sum_{m=-l}^{l}\left(a_l^mh_l^{(1)}(kr)+b_l^mk_l(kr)\right)Y_l^m(\theta,\varphi),~~~~{\rm if}~d=3, 
	\en
	where $a_m,b_m,a_l^m,b_l^m$ are constants. 
\end{theorem}
\begin{remark}\label{remark3.6}
	Here, for later use, we note that if $u$ possesses the series expansion {\rm (\ref{3.3})} or {\rm (\ref{3.35})}, then by direct calculation we see that if $d=2$, 
	\ben
	  &&\Delta u(r,\theta)=\sum\limits_{m\in\Z}k^2\left(-a_mH_m^{(1)}(kr)+b_mK_m(kr)\right)e^{im\theta}, \\
	  &&\pa_nu|_{\pa B_R}=\sum\limits_{m\in\Z}k\left(a_mH_m^{(1)\prime}(kR)+b_mK_m'(kR)\right)e^{im\theta}, \\
	  &&\pa_n\Delta u|_{\pa B_R}=\sum\limits_{m\in\Z}k^3\left(-a_mH_m^{(1)\prime}(kR)+b_mK_m'(kR)\right)e^{im\theta}, 
	\enn
	and if $d=3$, 
	\ben
	  &&\Delta u^s(r,\theta,\varphi)=\sum\limits_{l\in\N}\sum_{m=-l}^{l}k^2\left(-a_l^mh_l^{(1)}(kr)+b_l^mk_l(kr)\right)Y_l^m(\theta,\varphi), \\
	  &&\pa_nu|_{\pa B_R}=\sum\limits_{l\in\N}\sum_{m=-l}^{l}k\left(a_l^mh_l^{(1)\prime}(kr)+b_l^mk_l'(kr)\right)Y_l^m(\theta,\varphi), \\
	  &&\pa_n\Delta u|_{\pa B_R}=\sum\limits_{l\in\N}\sum_{m=-l}^{l}k^3\left(-a_l^mh_l^{(1)\prime}(kr)+b_l^mk_l'(kr)\right)Y_l^m(\theta,\varphi). 
	\enn
\end{remark}
\begin{lemma}\label{lem3.7}
	Suppose $u^s\in C^4(\R^d\setminus\ov\Om)\cap C^3(\R^d\setminus\Om)$ satisfies $\Delta^2u^s-k^4u^s=0$ in $\R^d\setminus\ov\Om$ and the radiation condition {\rm (\ref{1.2})}, then 
	\ben
	  k^2\I\int_{\pa\Om}(u^s\pa_n\Delta\ov u^s+\Delta u^s\pa_n\ov u^s)ds=-\I\int_{\pa\Om}(\Delta u^s\pa_n\Delta\ov u^s+k^4u^s\pa_n\ov u^s)ds
	\enn
\end{lemma}
\begin{proof}
	It suffices to show $\displaystyle\I\int_{\pa\Om}(\Delta u^s+k^2u^s)\pa_n(\Delta\ov u^s+k^2\ov u^s)ds=0$. Applying Green's theorem in $B_R\setminus\ov\Om$ we obtain that 
	\ben
	  &&\int_{\pa B_R}(\Delta u^s+k^2u^s)\pa_n(\Delta\ov u^s+k^2\ov u^s)ds-\int_{\pa\Om}(\Delta u^s+k^2u^s)\pa_n(\Delta\ov u^s+k^2\ov u^s)ds \\
	  =\;&&\int_{B_R\setminus\ov\Om}(\Delta u^s+k^2u^s)(\Delta^2\ov u^s+k^2\Delta\ov u^s)dx+\int_{B_R\setminus\ov\Om}|\na(\Delta u^s+k^2u^s)|^2dx \\
	  =\;&&\int_{B_R\setminus\ov\Om}k^2|\Delta u^s+k^2u^s|^2dx+\int_{B_R\setminus\ov\Om}|\na(\Delta u^s+k^2u^s)|^2dx, 
	\enn
	which implies 
	\ben
	  \I\int_{\pa B_R}(\Delta u^s+k^2u^s)\pa_n(\Delta\ov u^s+k^2\ov u^s)ds=\I\int_{\pa\Om}(\Delta u^s+k^2u^s)\pa_n(\Delta\ov u^s+k^2\ov u^s)ds. 
	\enn
	Now by Theorem \ref{thm3.5}, for $x\in\R^d\setminus B_R$, 
	\ben
	&&u^s(r,\theta)=\sum\limits_{m\in\Z}\left(a_mH_m^{(1)}(kr)+b_mK_m(kr)\right)e^{im\theta},~~~~{\rm if}~d=2, \\ 
	&&u^s(r,\theta,\varphi)=\sum\limits_{l\in\N}\sum_{m=-l}^{l}\left(a_l^mh_l^{(1)}(kr)+b_l^mk_l(kr)\right)Y_l^m(\theta,\varphi),~~~~{\rm if}~d=3, 
	\enn
	with $a_m,b_m,a_l^m,b_l^m$ constants. From Remark \ref{remark3.6} we see that 
	\ben
	  (\Delta u^s+k^2u^s)|_{\pa B_R}&&=
	  \left\{
	  \begin{array}{ll}
	  	\displaystyle\sum\limits_{m\in\Z}2k^2b_mK_m(kR)e^{im\theta},~~~&d=2,\\ [2mm]
	  	\displaystyle\sum\limits_{l\in\N}\sum_{m=-l}^{l}2k^2b_l^mk_l(kr)Y_l^m(\theta,\varphi),~&d=3, 
	  \end{array}
	  \right.
	  \\
	  \pa_n(\Delta u^s+k^2u^s)|_{\pa B_R}&&=
	   \left\{
	  \begin{array}{ll}
	  	\displaystyle\sum\limits_{m\in\Z}2k^3b_mK_m'(kR)e^{im\theta},~~~&d=2,\\ [2mm]
	  	\displaystyle\sum\limits_{l\in\N}\sum_{m=-l}^{l}2k^3b_l^mk_l'(kr)Y_l^m(\theta,\varphi),~&d=3, 
	  \end{array}
	  \right.
	\enn
	which further indicates 
	\ben
	  \int_{\pa B_R}(\Delta u^s+k^2u^s)\pa_n(\Delta\ov u^s+k^2\ov u^s)ds 
	  =
	  \left\{
	  \begin{array}{ll}
	  	\displaystyle2\pi R\sum\limits_{m\in\Z}4k^5|b_m|^2K_m(kR)\ov{K_m'(kR)},\\ [2mm]
	  	\displaystyle4\pi R^2\sum\limits_{l\in\N}\sum_{m=-l}^{l}4k^5|b_l^m|^2k_l(kr)\ov{k_l'(kr)}, 
	  \end{array}
	  \right.
	\enn
	for $d=2,3$. Since $K_m(r),k_l(r)$ take real values for $r>0$, we deduce that 
	\ben
	  \I\int_{\pa B_R}(\Delta u^s+k^2u^s)\pa_n(\Delta\ov u^s+k^2\ov u^s)ds=0. 
	\enn
	The proof is thus complete. 
\end{proof}
\begin{corollary}\label{cor3.8}
	Under the assumptions in Lemma {\rm \ref{lem3.7}}, if in addition 
	\ben
	  \I\int_{\pa\Om}(u^s\pa_n\Delta\ov u^s+\Delta u^s\pa_n\ov u^s)ds\leq0, 
	\enn
	then $\Delta u^s-k^2u^s=0$ in $\R^d\setminus\ov\Om$. 
\end{corollary}
\begin{proof}
	This would be a immediate consequence of (\ref{3.2}), Theorem \ref{thm3.4} and Lemma \ref{lem3.7}. 
\end{proof}
\begin{theorem}\label{thm3.9}
	Suppose $u^s\in C^4(\R^d\setminus\ov\Om)\cap C^3(\R^d\setminus\Om)$ satisfies $\Delta^2u^s-k^4u^s=0$ in $\R^d\setminus\ov\Om$ and the radiation condition {\rm (\ref{1.2})}, if further 
	\ben
	  \I\int_{\pa\Om}(u^s\pa_n\Delta\ov u^s+\Delta u^s\pa_n\ov u^s)ds\leq0~~{\rm and}~~\Rt\int_{\pa\Om}u^s\pa_n\ov u^s\geq0, 
	\enn
	then $u^s=0$ in $\R^d\setminus\ov\Om$. 
\end{theorem}
\begin{proof}
	From Corollary \ref{cor3.8}, we have $\Delta u^s-k^2u^s=0$ in $\R^d\setminus\ov\Om$. Hence, it follows that 
	\ben
	  0~&&=\int_{B_R\setminus\ov\Om}(\Delta u^s-k^2u^s)\ov u^sdx \\
	  &&=\int_{\pa B_R}\pa_nu^s\ov u^sds-\int_{\pa\Om}\pa_nu^s\ov u^sds-\int_{B_R\setminus\ov\Om}(|\na u^s|^2+k^2|u^s|^2)dx \\
	  &&=\int_{\pa B_R}(\pa_nu^s-iku^s)\ov u^sds+ik\int_{\pa B_R}|u^s|^2ds-\int_{\pa\Om}\pa_nu^s\ov u^sds \\
	  &&\quad~-\int_{B_R\setminus\ov\Om}(|\na u^s|^2+k^2|u^s|^2)dx. 
	\enn
	Taking the real part of the equation yields 
	\ben
	  \int_{B_R\setminus\ov\Om}(|\na u^s|^2+k^2|u^s|^2)dx\leq\Rt\int_{\pa B_R}(\pa_nu^s-iku^s)\ov u^sds. 
	\enn
	Due to the radiation condition (\ref{1.2}) and Lemma \ref{lem3.1}, it is derived that 
	\ben
	  \lim\limits_{R\rightarrow\infty}\int_{\pa B_R}(\pa_nu^s-iku^s)\ov u^sds=0, 
	\enn
	which implies $u^s\in H^1(\R^d\setminus\ov\Om)$ and $\|u^s\|_{H^1(\R^d\setminus\ov\Om)}=0$. Thus $u^s=0$ in $\R^d\setminus\ov\Om$. 
\end{proof}
\begin{remark}\label{remark3.10}
	Since we already have $\Delta u^s-k^2u^s=0$ in $\R^d\setminus\ov\Om$, Theorem {\rm \ref{thm3.9}} still holds if the condition $\displaystyle\Rt\int_{\pa\Om}u^s\pa_n\ov u^s\geq0$ is replaced by $\displaystyle\Rt\int_{\pa\Om}\Delta u^s\pa_n\Delta\ov u^s\geq0$. 
\end{remark}

In the following, we define a new type far-field pattern for the biharmonic scattering. It is seen that the scattered field $u^s$ to problem (\ref{1.1}) can be spilted into two parts $u^s_-:=\Delta u^s-k^2u^s$ and $u^s_+:=\Delta u^s+k^2u^s$. For the first part $u^s_-$, it is the radiating solution to the Helmholtz equation and thus has the asymptotic behavior (see \cite{FD06,DR13})
\be\label{3.4}
  u^s_-(x)=\frac{e^{ik|x|}}{|x|^{\frac{d-1}{2}}}\left\{u^s_{-,\infty}(\hat x)+O\left(\frac{1}{|x|}\right)\right\},~~~|x|\rightarrow\infty, 
\en
uniformly in all directions $\hat x=x/|x|$, where $u^s_{-,\infty}$ is the well-known acoustic far-field pattern for $u^s_-$ and has the representation 
\be\label{3.5}
  u^s_{-,\infty}(\hat x)=\frac{ie^{-i\frac{d-1}{4}\pi}k^{\frac{d-3}{2}}}{2(2\pi)^{\frac{d-1}{2}}}\int_{\pa\Om}\left(u^s_-(y)\frac{\pa e^{-ik\hat x\cdot y}}{\pa n(y)}-\frac{\pa u^s_-}{\pa n}(y)e^{-ik\hat x\cdot y}\right)ds(y). 
\en
As for the second part $u^s_+$, it is the scattering solution of the modified Helmholtz equation and thus has some similar properties to $u^s_-$. 
\begin{theorem}\label{thm3.11}
	Suppose $v^s\in C^2(\R^d\setminus\ov\Om)\cap C^1(\R^d\setminus\Om)$ satisfies $\Delta v^s-k^2v^s=0$ in $\R^d\setminus\ov\Om$ and the classical Sommerfeld radiation condition {\rm (\ref{3.6})}, then it follows that 
	\be\label{3.7}
	  v^s(x)=\frac{e^{-k|x|}}{|x|^{\frac{d-1}{2}}}\left\{v^s_{\infty}(\hat x)+O\left(\frac{1}{|x|}\right)\right\},~~~|x|\rightarrow\infty, 
	\en
	uniformly in all directions $\hat x$ with 
	\be\label{3.8}
	   v^s_{\infty}(\hat x)=\frac{k^{\frac{d-3}{2}}}{2(2\pi)^{\frac{d-1}{2}}}\int_{\pa\Om}\left(v^s(y)\frac{\pa e^{k\hat x\cdot y}}{\pa n(y)}-\frac{\pa v^s}{\pa n}(y)e^{k\hat x\cdot y}\right)ds(y). 
	\en
\end{theorem}
\begin{proof}
	Similar to the acoustic scattering, the conclusion follows directly from the representation theorem for $v^s$, see details in \cite{FD06,DR13}. 
\end{proof}
\begin{theorem}\label{thm3.13}
	Under the conditions in Theorem {\rm \ref{thm3.11}}, if $v^s_{\infty}=0$, then $v^s=0$ in $\R^d\setminus\ov\Om$. 
\end{theorem}
\begin{proof}
	By Lemma \ref{lem3.12} and Parseval's equality, we see that 
	\ben
	  \int_{\pa B_R}|v^s|^2ds=
	  \left\{
	  \begin{array}{ll}
	  	\displaystyle2\pi R\sum\limits_{m\in\Z}|a_m|^2|K_m(kR)|^2,~~~&d=2,\\ [2mm]
	  	\displaystyle4\pi R^2\sum\limits_{l\in\N}\sum_{m=-l}^{l}|a_l^m|^2|k_l(kr)|^2,~&d=3. 
	  \end{array}
	  \right.
	\enn
	Further, it is obtained from (\ref{3.7}) that 
	\ben
	  \lim\limits_{R\rightarrow\infty}Re^{2kR}\int_{\pa B_R}|v^s|^2ds=0, 
	\enn
	which indicates that 
	\ben
	  \left\{
	  \begin{array}{ll}
	  	\displaystyle\lim\limits_{R\rightarrow\infty}R^2e^{2kR}|a_m|^2|K_m(kR)|^2=0,~~~&d=2,\\ [3mm]
	  	\displaystyle\lim\limits_{R\rightarrow\infty}R^3e^{2kR}|a_l^m|^2|k_l(kr)|^2=0,~&d=3. 
	  \end{array}
	  \right.
	\enn
	Due to the asymptotic expansion (\ref{2.3}) of $K_m$ and (\ref{2.8}), we conculde that $a_m=0$ for all $m$ or $a_l^m=0$ for all $m,l$. Therefore, $v^s=0$ outside a sufficiently large disk (ball) and hence $v^s=0$ in $\R^d\setminus\ov\Om$ by analyticity. 
\end{proof}

Now we know that $u^s_+$ has the expansion in the form of (\ref{3.7}) with $u^s_{+,\infty}$. We give our definition for the far-field pattern of the biharmonic scattering field. 
\begin{definition}\label{def3.14}
	Suppose $u^s\in C^4(\R^d\setminus\ov\Om)\cap C^3(\R^d\setminus\Om)$ satisfies $\Delta^2u^s-k^4u^s=0$ in $\R^d\setminus\ov\Om$ and the radiation condition {\rm (\ref{1.2})}. Let $u^s_-:=\Delta u^s-k^2u^s$ and $u^s_+:=\Delta u^s+k^2u^s$. Then 
	\ben
	u^s_+(x)=\frac{e^{-k|x|}}{|x|^{\frac{d-1}{2}}}\left\{u^s_{+,\infty}(\hat x)+O\left(\frac{1}{|x|}\right)\right\},~~~|x|\rightarrow\infty, 
	\enn
	uniformly in all directions $\hat x$ with 
	\ben
	u^s_{+,\infty}(\hat x)=\frac{k^{\frac{d-3}{2}}}{2(2\pi)^{\frac{d-1}{2}}}\int_{\pa\Om}\left(u^s_+(y)\frac{\pa e^{k\hat x\cdot y}}{\pa n(y)}-\frac{\pa u^s_+}{\pa n}(y)e^{k\hat x\cdot y}\right)ds(y), 
	\enn
	and $u^s_-$ has the asymptotic behavior {\rm (\ref{3.4})} and {\rm (\ref{3.5})}. Define $(u^s_{+,\infty},u^s_{-,\infty})$ to be the far-field pattern of $u^s$. 
\end{definition}
\begin{theorem}\label{thm3.15}
	Under the assumptions in Definition {\rm \ref{def3.14}}, if $(u^s_{+,\infty},u^s_{-,\infty})=(0,0)$, then $u^s=0$ in $\R^d\setminus\ov\Om$. 
\end{theorem}
\begin{proof}
	Straightly employ Rellich's Lemma and Theorem \ref{thm3.13}. 
\end{proof}
\begin{remark}\label{remark3.16}
	We note that all the results in this section can be extended to proper Sobolev spaces by the standard density arguments. Further, if only $u^s_{+,\infty}=0$ or $u^s_{-,\infty}=0$, in general it can not be deduced that $u^s=0$. Nevertheless, since $u^s_+$ decays exponentially at infinity, is seems have no sense to include $u^s_{+,\infty}$ in the far-field pattern of $u^s$, which is, however, shown to be necessary and useful at section \ref{sec6}. 
\end{remark}

\section{Well-posedness}\label{sec3.5}
In this section, we want to obtain the well-posedness of biharmonic obstacle scattering problem with Dirichlet boundary condition, while the uniqueness is given by Theorem \ref{thm3.9}. Basically, we turn problem (\ref{1.1}) into a couple of Helmholtz and modified Helmholtz equations and develop corresponding boundary integral equation method. 
Different from \cite{HP23}, the problem is then reduced to equivalent boundary integral equations by expressing the solutions as combined single- and double-layer potentials. The Fredholm properties of the relating integral operators are derived using the theory of pseudodifferential operators and then the well-posedness follows. 


Denote by $\Phi_k$ the fundamental solution of the Helmholtz equation in $\R^d$ with wave number $k$, which is 
\ben
   \Phi_k(x,y)=
   \left\{
   \begin{array}{ll}
   	\displaystyle\frac{i}{4}H_0^{(1)}(k|x-y|),~~~&d=2,\\ [3mm]
   	\displaystyle\frac{e^{ik|x-y|}}{4\pi |x-y|},~&d=3. 
   \end{array}
   \right.
\enn
We introduce the famous single- and double-layer potentials in the acoustic scattering, 
\ben
(SL_{k}\varphi)(x)&&:=\int_{\pa\Om}\Phi_k(x,y)\varphi(y)ds(y),~~~x\in\R^d\setminus\pa\Om, \\
(DL_{k}\varphi)(x)&&:=\int_{\pa\Om}\frac{\pa\Phi_k(x,y)}{\pa\nu(y)}\varphi(y)ds(y),~~~x\in\R^d\setminus\pa\Om. 
\enn
Also, we give the definitions of the associated boundary integral operators, 
\ben
(S_{k}\varphi)(x)&&:=\int_{\pa\Om}\Phi_k(x,y)\varphi(y)ds(y),~~~x\in\pa\Om, \\
(K_{k}\varphi)(x)&&:=\int_{\pa\Om}\frac{\pa\Phi_k(x,y)}{\pa\nu(y)}\varphi(y)ds(y),~~~x\in\pa\Om, \\
(K'_{k}\varphi)(x)&&:=\int_{\pa\Om}\frac{\pa\Phi_k(x,y)}{\pa\nu(x)}\varphi(y)ds(y),~~~x\in\pa\Om, \\
(T_{k}\varphi)(x)&&:=\frac{\pa}{\pa\nu(x)}\int_{\pa\Om}\frac{\pa\Phi_k(x,y)}{\pa\nu(y)}\varphi(y)ds(y),~~~x\in\pa\Om. 
\enn
Further, we remind the volume potential 
\ben
  (\wid{SL}_{k}\varphi)(x)&&:=\int_{\Om}\Phi_k(x,y)\varphi(y)dy,~~~x\in\R^d. 
\enn
The basic properties of these operators can be found in \cite{FD06,DR13}. 

We spilt the biharmonic equation into a couple of Helmholtz and modified Helmholtz equations. For $(u^s_+,u^s_-)$ we see that problem (\ref{1.1}) with $\mathcal{B}=\mathcal{B}_D$ becomes 
\be\label{5.1}
\left\{
\begin{array}{ll}
	\Delta u^s_+-k^2u^s_+=0,~~\Delta u^s_-+k^2u^s_-=0~~~&{\rm in}~\R^d\setminus\ov\Om, \\[2mm]
	u^s_+-u^s_-=2k^2f,~~\pa_nu^s_+-\pa_nu^s_-=2k^2g~~~&{\rm on}~\pa\Om, \\ [2mm]
	\pa_rw-ikw=o\left(r^{-\frac{d-1}{2}}\right),~~r=|x|\rightarrow\infty,~&w=u^s_\pm, 
\end{array}
\right.
\en
which is an interior tranmission problem in the exterior domain and can not be handled by the usual methods. It is also remarked that problem \eqref{1.1} with some boundary condition, such as $\mathcal{B}=\mathcal{B}_N$, after decomposition can be easily solved by the classical methods in the acoustic scattering. Here we require the boundary data $(f,g)\in H^{3/2}(\pa\Om)\times H^{1/2}(\pa\Om)$. We want to seek the solutions $u^s_\pm$ of problem (\ref{5.1}) in the form of 
\ben
  u^s_+=SL_{ik}\varphi-DL_{ik}\psi+i\eta SL_{ik}(S_0^2\psi),~~~u^s_-=SL_k\varphi-DL_k\psi~~{\rm in}~\R^d\setminus\ov\Om. 
\enn
with density $(\varphi,\psi)\in H^{-3/2}(\pa\Om)\times H^{-1/2}(\pa\Om)$ and the real parameter $\eta\neq0$. Now problem (\ref{5.1}) can be equivalently reduced to the following boundary integral equations on $\pa\Om$ (note that the mapping properties and the jump relations of these operators in weaker spaces can be found in \cite{AH13}): 
\be\label{5.2}
\left(
\begin{array}{cc}
	S_{ik}-S_k & -K_{ik}+K_k+i\eta S_{ik}S_0^2  \\
	-K_{ik}'+K_k' & T_{ik}-T_k+i\eta K_{ik}'S_0^2-\frac{i\eta}{2}S_0^2
\end{array}
\right)
\left(
\begin{array}{c}
	\varphi \\
	\psi
\end{array}
\right)
=
\left(
\begin{array}{c}
	2k^2f \\
	-2k^2g 
\end{array}
\right).
\en
Denote by $M(k)$ the matrix operator in the left hand side of (\ref{5.2}), and set 
\ben
  Z(k)=
  \left(
  \begin{array}{cc}
  	S_{ik}-S_k & -K_{ik}+K_k \\
  	-K_{ik}'+K_k' & T_{ik}-T_k
  \end{array}
  \right).
\enn
 From \cite{AH13} it is known that $Z(k):H^{-3/2}(\pa\Om)\times H^{-1/2}(\pa\Om)\rightarrow H^{3/2}(\pa\Om)\times H^{1/2}(\pa\Om)$ is bounded. Further, since $\pa\Om\in C^{3,\alpha}$, we deduce from \cite{AK89} that $M(k)-Z(k):H^{-3/2}(\pa\Om)\times H^{-1/2}(\pa\Om)\rightarrow H^{3/2}(\pa\Om)\times H^{1/2}(\pa\Om)$ is compact. We first consider the uniqueness of equation (\ref{5.2}). 
\begin{theorem}\label{thm5.1}
	The equation {\rm \eqref{5.2}} has at one solution, i.e., the operator $M(k)$ is injective. 
\end{theorem}
\begin{proof}
	Suppose $M(k)(\varphi,\psi)^T=(0,0)^T$. Define 
	\ben
	  v=SL_{ik}\varphi-DL_{ik}\psi+i\eta SL_{ik}(S_0^2\psi),~~~w=SL_k\varphi-DL_k\psi~~{\rm in}~\R^d\setminus\pa\Om. 
	\enn
	Then $(v,w)$ satisfy the homogeneous problem (\ref{5.1}), which indicates that $v=w=0$ in $\R^d\setminus\ov\Om$ by Theorem \ref{thm3.9}. By the notations $(\cdot)^\pm$ we denote the traces from the exterior (+) and interior (--) of the boundary $\pa\Om$. It follows that $v^+=\pa_nv^+=w^+=\pa_nw^+=0$ on $\pa\Om$, which further implies by the jump relations that 
	\ben
	  v^-=w^-=\psi,~~~\pa_nv^-=\varphi+i\eta S_0^2\psi,~~~\pa_nw^-=\varphi~~{\rm on}~\pa\Om. 
	\enn
	Since $\Delta v-k^2v=\Delta w+k^2w=0$ in $\Om$, we yields that 
	\ben
	  2i\eta\int_{\pa\Om}|S_0\psi|^2ds\;&&=2i\eta\int_{\pa\Om}\ov\psi S_0^2\psi ds \\
	  &&=\int_{\pa\Om}\left[\pa_n(v^--w^-)(\ov v^-+\ov w^-)-(v^--w^-)\pa_n(\ov v^-+\ov w^-)\right]ds \\
	  &&=\int_{\Om}\left[\Delta(v-w)(\ov v+\ov w)-(v-w)\Delta(\ov v+\ov w)\right]dx \\
	  &&=\int_{\Om}k^2\left(|v+w|^2-|v-w|^2\right)dx. 
	\enn
	Taking the imaginary part we obtain $S_0\psi=0$ on $\pa\Om$ and hence $\psi=0$. Finally, from $v^+=\psi=0$ on $\pa\Om$, we deduce that $S_{ik}\varphi=0$ on $\pa\Om$, which also leads to $\varphi=0$. The proof is thus complete. 
\end{proof}

Next we study the Fredholm property of the operator $M(k)$. Consider an auxiliary operator 
\ben
\hat Z(k)=\frac{2}{3}
\left(
\begin{array}{ccc}
	S_{i2|k|}-S_{i|k|} & -K_{i2|k|}+K_{i|k|}  \\
	-K_{i2|k|}'+K_{i|k|}' & T_{i2|k|}-T_{i|k|}
\end{array}
\right). 
\enn
We refer a significant property of this operator. 
\begin{theorem}{\rm ({\cite[Lemma 3.7]{AH13}})}\label{thm5.2}
	$\hat{Z}(k): H^{-3/2}(\pa\Om)\times H^{-1/2}(\pa\Om)\rightarrow H^{3/2}(\pa\Om)\times H^{1/2}(\pa\Om)$ is coercive. 
\end{theorem}
\begin{lemma}\label{lem5.3}
	The pseudodifferential operator $\wid{SL}_{ik}-\wid{SL}_{k}-2/3(\wid{SL}_{i2|k|}-\wid{SL}_{i|k|})$ is of order {\rm -5}. 
\end{lemma}
\begin{proof}
	We first consider the two-dimensional case. It is seen that 
	\ben
	  \left(\wid{SL}_{ik}-\wid{SL}_{k}-2/3(\wid{SL}_{i2|k|}-\wid{SL}_{i|k|})\right)\varphi(x)=\int_{\Om}a(x,x-y)\varphi(y)dy 
	\enn
	with the kernel  
	\ben
	  a(x,z):=\frac{i}{4}\left(H_0^{(1)}(ik|z|)-H_0^{(1)}(k|z|)-\frac{2}{3}\left(H_0^{(1)}(i2|kz|)-H_0^{(1)}(i|kz|)\right)\right). 
	\enn
	From the proof of \cite[Theorem 3.2]{AH13} we deduce that 
	\ben
	  \frac{i}{4}\left(H_0^{(1)}(ik|z|)-H_0^{(1)}(k|z|)\right)&&=f(x,z)+\sum_{j=0}^{\infty}p_{j+2}(x,z)\ln|z|, \\
	  \frac{i}{4}\left(H_0^{(1)}(i2|kz|)-H_0^{(1)}(i|kz|)\right)&&=\wid f(x,z)+\sum_{j=0}^{\infty}\wid p_{j+2}(x,z)\ln|z|, 
	\enn
	where $f,\wid f\in C^\infty(\Om\times\R^2)$ and 
	\ben
	p_{j+2}(x,z)=
	\left\{
	\begin{array}{ll}
		0 & \text{if $j$ is odd}, \\
		\displaystyle\frac{1}{2\pi}\frac{(-1)^{p+1}}{(p+1)!^2}(i^{j+2}-1)k^{j+2}\left(\frac{|z|}{2}\right)^{j+2} & \text{if $j=2p$}, 
	\end{array}
	\right.
	\enn
	\ben
	\quad\wid p_{j+2}(x,z)=
	\left\{
	\begin{array}{ll}
		0 & \text{if $j$ is odd}, \\
		\displaystyle\frac{1}{2\pi}\frac{(-1)^{p+1}}{(p+1)!^2}(2^{j+2}-1)(ik)^{j+2}\left(\frac{|z|}{2}\right)^{j+2} & \text{if $j=2p$}. 
	\end{array}
	\right.
	\enn
	Hence, by direct calculation we derive that 
	\ben
	  a(x,z)=\hat f(x,z)+\sum_{j=0}^{\infty}\hat p_{j+4}(x,z)\ln|z| 
	\enn
	with $\hat f\in C^\infty(\Om\times\R^2)$ and $\hat p_{j+4}(x,z)=0$ if $j$ is odd and 
	\ben
	  \hat p_{j+4}(x,z)=\frac{1}{2\pi}\frac{(-1)^{p}}{(p+2)!^2}\left(i^j-1-\frac{2}{3}(2^{j+4}-1)i^j\right)k^{j+4}\left(\frac{|z|}{2}\right)^{j+4}
	\enn
	for $j=2p$. Note that the functions $\hat p_q$ satisfy $\hat p_q(x,tz)=t^q\hat p_q(x,z)$, which implies that the kernel of $\wid{SL}_{ik}-\wid{SL}_{k}-2/3(\wid{SL}_{i2|k|}-\wid{SL}_{i|k|})$ is a pseudohomogeneous kernel of degree 4. Thus by \cite[Theorem 7.1.1]{GW08} we yield that $\wid{SL}_{ik}-\wid{SL}_{k}-2/3(\wid{SL}_{i2|k|}-\wid{SL}_{i|k|})$ is a pseudodifferential operator of order -6 (thus of order -5). 
	
	Now we consider the three-dimensional case. The corresponding kernel is 
	\ben
	  a(x,z)=\frac{1}{4\pi|z|}\left(e^{-k|z|}-e^{ik|z|}-\frac{2}{3}(e^{-2k|z|}-e^{-k|z|})\right). 
	\enn
	Computing directly from the series expansion of the exponential function we obtain that 
	\ben
	  a(x,z)=-\frac{1+3i}{12\pi}k+\sum_{j=0}^\infty a_{j+2}(x,z), 
	\enn
	where 
	\ben
	  a_{j+2}(x,z)=\frac{-i^{j+1}}{4\pi(j+3)!}\left(i^{j+3}-1-\frac{2}{3}(2^{j+3}-1)i^{j+3}\right)k^{j+3}|z|^{j+2},~~~{\rm for}~j\in\N. 
	\enn
	Since $a_p(x,tz)=t^pa_p(x,z)$, $a$ is a pseudohomogeneous kernel of degree 2, which indicates that $\wid{SL}_{ik}-\wid{SL}_{k}-2/3(\wid{SL}_{i2|k|}-\wid{SL}_{i|k|})$ is a pseudodifferential operator of order -5. 
\end{proof}
\begin{theorem}\label{thm5.4}
	$Z(k)-\hat Z(k):H^{-3/2}(\pa\Om)\times H^{-1/2}(\pa\Om)\rightarrow H^{3/2}(\pa\Om)\times H^{1/2}(\pa\Om)$ is compact. 
\end{theorem}
\begin{proof}
	Combining Lemma \ref{lem5.3} and \cite[Theorem 8.5.8]{GW08}, we obtain that $SL_{ik}-SL_{k}-2/3(SL_{i2|k|}-SL_{i|k|}):H^{-3/2}(\pa\Om)\rightarrow H^3(\Om)$ is bounded. Since $DL_k\varphi=-\na SL_k(\varphi n)$, it follows that $DL_{ik}-DL_{k}-2/3(DL_{i2|k|}-DL_{i|k|}):H^{-1/2}(\pa\Om)\rightarrow H^3(\Om)$ is bounded. Then by the classical trace theorems we deduce the conclusion. 
\end{proof}
\begin{theorem}\label{thm5.5}
	Given $(f,g)\in H^{3/2}(\pa\Om)\times H^{1/2}(\pa\Om)$, then the following problem 
	\ben
	\left\{
	\begin{array}{ll}
		\Delta^2 u^s-k^4u^s=0~~~&{\rm in}~\R^d\setminus\ov\Om, \\[2mm]
		\mathcal{B}_D(u^s)=(f,g)~~~&{\rm on}~\pa\Om, \\ [2mm]
		\pa_rw-ikw=o\left(r^{-\frac{d-1}{2}}\right),~~r=|x|\rightarrow\infty,~&w=u^s,\Delta u^s, 
	\end{array}
	\right.
	\enn
	has a unique solution $u^s\in H^2_{loc}(\R^d\setminus\ov\Om)$ such that 
	\ben
	  \|u^s\|_{H^2(B_R\setminus\ov\Om)}\leq C\left(\|f\|_{H^{3/2}(\pa\Om)}+\|g\|_{H^{1/2}(\pa\Om)}\right), 
	\enn
	where $R>0$ is sufficiently large and $C$ is a positive constant depending only on $R$. 
\end{theorem}
\begin{proof}
	From Theorems \ref{thm5.1}, \ref{thm5.2} and \ref{thm5.4}, we see that problem (\ref{5.1}) has a unique solution $(u^s_+,u^s_-)=(\Delta u^s+k^2u^s,\Delta u^s-k^2u^s)$ in the form of 
	\ben
	  u^s_+=SL_{ik}\varphi-DL_{ik}\psi+i\eta SL_{ik}(S_0^2\psi),~~~u^s_-=SL_k\varphi-DL_k\psi~~{\rm in}~\R^d\setminus\ov\Om, 
	\enn
	where $(\varphi,\psi)$ is the unique solution to the boundary integral equation (\ref{5.2}) such that 
	\ben
	  \|\varphi\|_{H^{-3/2}(\pa\Om)}+\|\psi\|_{H^{-1/2}(\pa\Om)}\leq C\left(\|f\|_{H^{3/2}(\pa\Om)}+\|g\|_{H^{1/2}(\pa\Om)}\right)
	\enn
	with $C>0$ a constant. By \cite[Corollary 3.3]{AH13}, we further have that $2k^2u^s=u^s_+-u^s_-=(SL_{ik}-SL_k)\varphi-(DL_{ik}-DL_k)\psi+i\eta SL_{ik}(S_0^2\psi)\in H^2(B_R\setminus\ov\Om)$ and 
	\ben
	  \|u^s\|_{H^2(B_R\setminus\ov\Om)}~&&\leq C\left(\|(SL_{ik}-SL_k)\varphi\|_{H^2(B_R\setminus\ov\Om)}+\|(DL_{ik}-DL_k)\psi\|_{H^2(B_R\setminus\ov\Om)}\right. \\
	  &&\left.\qquad~+\|SL_{ik}(S_0^2\psi)\|_{H^2(B_R\setminus\ov\Om)}\right) \\
	  &&\leq C\left(\|\varphi\|_{H^{-3/2}(\pa\Om)}+\|\psi\|_{H^{-1/2}(\pa\Om)}\right) \\
	  &&\leq C\left(\|f\|_{H^{3/2}(\pa\Om)}+\|g\|_{H^{1/2}(\pa\Om)}\right), 
	\enn
	which is the desired a priori estimate. 
\end{proof}
\begin{remark}\label{remark5.6}
	For problem {\rm (\ref{1.1})} with $\mathcal{B}(u)=(\Delta u,\pa_n\Delta u)$, it can be verified that $(u^s_+,-u^s_-)$ satisfies problem {\rm (\ref{5.1})}. Therefore, following the same process, we can solve problem {\rm (\ref{1.1})} in the case $\mathcal{B}(u)=(\Delta u,\pa_n\Delta u)$. 
\end{remark}

\section{The inverse problem}\label{sec6}
\setcounter{equation}{0}
In this section, we study the inverse problem of uniquely determining the obstacle $\Om$ from measurements at a fixed frequency, which relies heavily on the reciprocity relations of our new type far-field pattern and the scattered field. 


 Denote by $u^s(x,a,b)$ the solutions to problem (\ref{1.1}) corresponding to the incident wave $u^i(x,\hat y,b):=e^{ibx\cdot\hat y}$ or $u^i(x,y,b):=\Phi_{b}(x,y)$ with $a=\hat y\in\Sp^{d-1}$ or $y\in\R^d\setminus\ov\Om$ and $b=k,ik$, respectively. Let $u(x,a,b)=u^s(x,a,b)+u^i(x,a,b)$. 
We begin with exploring some reciprocity relations for the far-field pattern and the scattered field. 
\begin{theorem}\label{thm6.3}
	For biharmonic obstacle $\Om$ with $\mathcal{B}=\mathcal{B}_D$, we have the reciprocity relations 
	 \ben
	 &&\quad\left(
	 \begin{array}{cc}
	 	u^s_+(y,\hat x,k) & u^s_-(y,\hat x,k) \\
	 	u^s_+(y,\hat x,ik) & u^s_-(y,\hat x,ik)
	 \end{array}
	 \right) \\
	 &&=\frac{2(2\pi)^{\frac{d-1}{2}}}{ik^{\frac{d-3}{2}}}
	\left(
	\begin{array}{cc}
		e^{i\frac{d-1}{4}\pi}u^s_{-,\infty}(-\hat x,y,ik) & e^{i\frac{d-1}{4}\pi}u^s_{-,\infty}(-\hat x,y,k) \\
		iu^s_{+,\infty}(-\hat x,y,ik) & iu^s_{+,\infty}(-\hat x,y,k)
	\end{array}
	\right) 
	\enn
	provided that $y\in\R^d\setminus\ov\Om$ and $\hat x\in\Sp^{d-1}$. 
\end{theorem}
\begin{proof}
	By the boundary condition we see that 
	\be\label{6.1}\nonumber
	  \int_{\pa\Om}&&\left(u(\cdot,\hat x,b_1)\pa_n\Delta u(\cdot,y,b_2)+\Delta u(\cdot,\hat x,b_1)\pa_nu(\cdot,y,b_2)\right. \\ 
	  &&~\left.-\pa_n\Delta u(\cdot,\hat x,b_1)u(\cdot,y,b_2)-\pa_nu(\cdot,\hat x,b_1)\Delta u(\cdot,y,b_2)\right)ds=0 
	\en
	for $b_1,b_2=k,ik$. In the following, we only consider the case $(b_1,b_2)=(k,k)$, the other three cases can be managed analogously. 
	
	Integration by parts over $\Om$ yields 
	\be\label{6.2}\nonumber
	  \int_{\pa\Om}&&\left(u^i(\cdot,\hat x,k)\pa_n\Delta u^i(\cdot,y,k)+\Delta u^i(\cdot,\hat x,k)\pa_nu^i(\cdot,y,k)\right. \\ 
	  &&~\left.-\pa_n\Delta u^i(\cdot,\hat x,k)u^i(\cdot,y,k)-\pa_nu^i(\cdot,\hat x,k)\Delta u^i(\cdot,y,k)\right)ds=0. 
	\en
	From Green's theorem, the radiation condition (\ref{1.2}) and Lemma \ref{lem3.1}, we obtain that 
	\be\label{6.3}\nonumber
	\int_{\pa\Om}&&\left(u^s(\cdot,\hat x,k)\pa_n\Delta u^s(\cdot,y,k)+\Delta u^s(\cdot,\hat x,k)\pa_nu^s(\cdot,y,k)\right. \\
	&&~\left.-\pa_n\Delta u^s(\cdot,\hat x,k)u^s(\cdot,y,k)-\pa_nu^s(\cdot,\hat x,k)\Delta u^s(\cdot,y,k)\right)ds=0. 
	\en
	Combining (\ref{6.1})-(\ref{6.3}), since $u=u^i+u^s$, it is derived that 
	\ben
	  \int_{\pa\Om}&&\left(u^s(\cdot,\hat x,k)\pa_n\Delta u^i(\cdot,y,k)+\Delta u^s(\cdot,\hat x,k)\pa_nu^i(\cdot,y,k)\right. \\
	  &&~\left.-\pa_n\Delta u^s(\cdot,\hat x,k)u^i(\cdot,y,k)-\pa_nu^s(\cdot,\hat x,k)\Delta u^i(\cdot,y,k)\right)ds \\
	  =-\int_{\pa\Om}&&\left(u^i(\cdot,\hat x,k)\pa_n\Delta u^s(\cdot,y,k)+\Delta u^i(\cdot,\hat x,k)\pa_nu^s(\cdot,y,k)\right. \\
	  &&~\left.-\pa_n\Delta u^i(\cdot,\hat x,k)u^s(\cdot,y,k)-\pa_nu^i(\cdot,\hat x,k)\Delta u^s(\cdot,y,k)\right)ds. 
	\enn
	Note that $u^s=1/2k^2(u^s_+-u^s_-)$, $\Delta u^s=1/2(u^s_++u^s_-)$ and $u^i(\cdot,a,k)$ satisfies the Helmholtz equation for $a=\hat x,y$. By the representation theorem, we have 
	\ben
	  \int_{\pa\Om}&&\left(u^s(\cdot,\hat x,k)\pa_n\Delta u^i(\cdot,y,k)+\Delta u^s(\cdot,\hat x,k)\pa_nu^i(\cdot,y,k)\right. \\
	  &&~\left.-\pa_n\Delta u^s(\cdot,\hat x,k)u^i(\cdot,y,k)-\pa_nu^s(\cdot,\hat x,k)\Delta u^i(\cdot,y,k)\right)ds \\
	  =\int_{\pa\Om}&&\left(u^s_-(\cdot,\hat x,k)\pa_nu^i(\cdot,y,k)-u^i(\cdot,y,k)\pa_nu^s_-(\cdot,\hat x,k)\right)ds=u^s_-(y,\hat x,k). 
	\enn
	Moreover, from the expression of the far-field pattern (\ref{3.5}), we deduce that 
	\ben
	  &&\quad-\int_{\pa\Om}\left(u^i(\cdot,\hat x,k)\pa_n\Delta u^s(\cdot,y,k)+\Delta u^i(\cdot,\hat x,k)\pa_nu^s(\cdot,y,k)\right. \\
	  &&\qquad\quad\quad\;\left.-\pa_n\Delta u^i(\cdot,\hat x,k)u^s(\cdot,y,k)-\pa_nu^i(\cdot,\hat x,k)\Delta u^s(\cdot,y,k)\right)ds \\
	  &&=\int_{\pa\Om}\left(u^s_-(\cdot,y,k)\pa_nu^i(\cdot,\hat x,k)-u^i(\cdot,\hat x,k)\pa_nu^s_-(\cdot,y,k)\right)ds \\ [2mm]
	  &&=\frac{2(2\pi)^{\frac{d-1}{2}}}{ie^{-i\frac{d-1}{4}\pi}k^{\frac{d-3}{2}}}u^s_{-,\infty}(-\hat x,y,k),  
	\enn
	which implies 
	\ben
	  u^s_-(y,\hat x,k)=\frac{2(2\pi)^{\frac{d-1}{2}}}{ie^{-i\frac{d-1}{4}\pi}k^{\frac{d-3}{2}}}u^s_{-,\infty}(-\hat x,y,k). 
	\enn
	The other three reciprocity relations can be obtained from other three choices of $(b_1,b_2)$, i.e., $(b_1,b_2)=(k,ik),(ik,k)$ and $(ik,ik)$. The proof is thus complete. 
\end{proof}

The next two theorems about other kinds reciprocity relations follow closely as Theorem \ref{thm6.3}, the detailed proof is omitted. 
\begin{theorem}\label{thm6.4}
	The reciprocity relations hold 
	 \ben
	&&\quad\left(
	\begin{array}{cc}
		e^{i\frac{d-1}{4}\pi}u^s_{+,\infty}(\hat x,\hat y,k) & u^s_{-,\infty}(\hat x,\hat y,k) \\
		u^s_{+,\infty}(\hat x,\hat y,ik) & u^s_{-,\infty}(\hat x,\hat y,ik)
	\end{array}
	\right) \\
	&&=
	\left(
	\begin{array}{cc}
		u^s_{-,\infty}(-\hat y,-\hat x,ik) & u^s_{-,\infty}(-\hat y,-\hat x,k) \\
		u^s_{+,\infty}(-\hat y,-\hat x,ik) & e^{i\frac{d-1}{4}\pi}u^s_{+,\infty}(-\hat y,-\hat x,k)
	\end{array}
	\right)
	\enn
	for biharmonic obstacle $\Om$ with $\mathcal{B}=\mathcal{B}_D$ and $\hat x,\hat y\in\Sp^{d-1}$.
\end{theorem}
\begin{theorem}\label{thm6.5}
	For biharmonic obstacle $\Om$ with $\mathcal{B}=\mathcal{B}_D$, we have the symmetry relations 
	\ben
	\left(
	\begin{array}{cc}
		u^s_{+}(x,y,k) & u^s_{-}(x,y,k) \\
		u^s_{+}(x,y,ik) & u^s_{-}(x,y,ik)
	\end{array}
	\right)
	=
	\left(
	\begin{array}{cc}
		u^s_{-}(y,x,ik) & u^s_{-}(y,x,k) \\
		u^s_{+}(y,x,ik) & u^s_{+}(y,x,k)
	\end{array}
	\right), 
	\enn
	where $x,y\in\R^d\setminus\ov\Om$. 
\end{theorem}
\begin{remark}\label{remark6.6}
	It can be easily seen that Theorems {\rm \ref{thm6.3}-\ref{thm6.5}} also hold for $\mathcal{B}=\mathcal{B}_N$, $\mathcal{B}(u)=(\pa_nu,\pa_n\Delta u),(\Delta u,\pa_n\Delta u)$. 
\end{remark}

We now are at the position to establish uniqueness results for the inverse biharmonic obstacle scattering problems in the Dirichlet case $\mathcal{B}=\mathcal{B}_D$. 
\begin{theorem}\label{thm6.7}
	Suppose $\Om$ and $\wid\Om$ are two biharmonic obstacles with $\mathcal{B}=\mathcal{B}_D$. If, for all $\hat x,\hat y\in\Sp^{d-1}$, one of the following holds{\rm :} \\
	{\rm (i)}$(u^s_{+,\infty}(\hat x,\hat y,k),u^s_{+,\infty}(\hat x,\hat y,ik))=(\wid u^s_{+,\infty}(\hat x,\hat y,k),\wid u^s_{+,\infty}(\hat x,\hat y,ik))${\rm ;} \\
	{\rm (ii)}$(u^s_{-,\infty}(\hat x,\hat y,k),u^s_{-,\infty}(\hat x,\hat y,ik))=(\wid u^s_{-,\infty}(\hat x,\hat y,k),\wid u^s_{-,\infty}(\hat x,\hat y,ik))${\rm ;} \\
	then $\Om=\wid\Om$. 
\end{theorem}
\begin{proof}
	The proof is standard as in \cite{AR93}. We only consider the situation that (i) is satisfied. Denote by $G$ the unbounded connected part of $\R^d\setminus\overline{(\Om\cup\wid\Om)}$. By Theorem \ref{thm3.13}, we know that $(u^s_{+}(x,\hat y,k),u^s_{+}(x,\hat y,ik))=(\wid u^s_{+}(x,\hat y,k),\wid u^s_{+}(x,\hat y,ik))$ for all $x\in G$ and $\hat y\in\Sp^{d-1}$. Due to the reciprocity relations Theorem \ref{thm6.3}, we see that $(u^s_{+,\infty}(\hat y,x,ik),u^s_{-,\infty}(\hat y,x,ik))=(\wid u^s_{+,\infty}(\hat y,x,ik),\wid u^s_{-,\infty}(\hat y,x,ik))$ with $\hat y\in\Sp^{d-1}$ and $x\in G$. Furthermore, by Theorem \ref{thm3.15}, it immediately follows that $u^s(x,y,ik)=\wid u^s(x,y,ik)$ for all $x,y\in G$. 
	
	Now suppose $\Om\neq\wid\Om$. Then, without loss of generality, there exists a $x_0\in\pa G$ such that $x_0\in\pa\Om$ and $x_0\notin\ov{\wid\Om}$. Choose $\delta>0$ such that the sequence 
	\ben
	  x_j:=x_0+\frac{\delta}{j}n(x_0)\in B\cap G,~~~j\in\N, 
	\enn
	where $B$ is a small ball centered at $x_0$ satisfying $B\cap\ov{\wid\Om}=\emptyset$. We know that $u^s(x,x_j,ik)=\wid u^s(x,x_j,ik)$ for $x\in G$ and $j\in\N$. From the positive distance between $x_j$ and $\wid\Om$, since $k^4$ is not a biharmonic Dirichlet eigenvalue in $\wid\Om$, by Theorem \ref{thm5.5} it yields that $\wid u^s(x,x_j,ik)$ is uniformly bounded in $H^2(B\cap\Om)$ for $j\in\N$, which further indicates by the trace theorem that $\|\wid u^s(x,x_j,ik)\|_{H^{3/2}(B\cap\pa\Om)}\leq C$ for all $j$ and some positive constant $C$. On the other hand, since $u^s(x,x_j,ik)=\wid u^s(x,x_j,ik)$ for $x\in G$, from the boundary condition on $\pa\Om$ we have that 
	\ben
	  \|\wid u^s(x,x_j,ik)\|_{H^{3/2}(B\cap\pa\Om)}=\|\Phi_{ik}(x,x_j)\|_{H^{3/2}(B\cap\pa\Om)}\rightarrow\infty,~~~n\rightarrow\infty, 
	\enn
	which is a contradiction. Therefore, $\Om=\wid\Om$. 
\end{proof}
\begin{theorem}\label{thm6.8}
	Let $\Om$ and $\wid\Om$ be two biharmonic obstacles with $\mathcal{B}=\mathcal{B}_D$ and $\ov{\Om\cup\wid\Om}\subset B_R$ for some $R>0$. If, for all $x,y\in\pa B_R$, one of the following is satisfied{\rm :} \\
	{\rm (i)}$(u^s_{+}(x,y,k),u^s_{+}(x,y,ik))=(\wid u^s_{+}(x,y,k),\wid u^s_{+}(x,y,ik))${\rm ;} \\
	{\rm (ii)}$(u^s_{-}(x,y,k),u^s_{-}(x,y,ik))=(\wid u^s_{-}(x,y,k),\wid u^s_{-}(x,y,ik))${\rm ;} \\
	then $\Om=\wid\Om$. 
\end{theorem}
\begin{proof}
	Here, for simplicity, we only prove case (ii). Also denote by $G$ the unbounded connected part of $\R^d\setminus\overline{(\Om\cup\wid\Om)}$. Due to the uniqueness of the exterior Dirichlet Helmholtz equation, we deduce from (ii) that  $u^s_-(x,y,b)=\wid u^s_-(x,y,b)$ for all $x\in\R^2\setminus B_R$, $y\in\pa B_R$ and $b=k,ik$. Then by analyticity we know $u^s_-(x,y,b)=\wid u^s_-(x,y,b)$ for all $x\in G$, $y\in\pa B_R$ and $b=k,ik$. Further, the symmetry relations in Theorem \ref{thm6.5} yield that $u^s_\pm(y,x,k)=\wid u^s_\pm(y,x,k)$ for all $y\in\pa B_R$ and $x\in G$, which implies $(u^s(y,x,k),\Delta u^s(y,x,k))=(\wid u^s(y,x,k),\Delta\wid u^s(y,x,k))$ with $y\in\pa B_R$ and $x\in G$. From the uniqueness of problem \eqref{1.1} with $\mathcal{B}=\mathcal{B}_N$ in Theorem \ref{thm3.9}, it is obtained that $u^s(y,x,k)=\wid u^s(y,x,k)$ for $y\in\R^d\setminus B_R$ and $x\in G$. By analyticity, we further deduce $u^s(y,x,k)=\wid u^s(y,x,k)$ for $x,y\in G$ . Finally, following the same line as the proof of Theorem \ref{thm6.7} we can conclude that $\Om=\wid\Om$. 
\end{proof}
\begin{remark}\label{remark6.9}
	Though the part $u^s_+$ decays exponentially at infinity and thus hard to measure at practice, it is sufficient for the uniqueness results to hold that only the condition {\rm (ii)} in Theorems {\rm \ref{thm6.7}} and {\rm \ref{thm6.8}} is fullfilled, while $u^s_-$ and $u^s_{-,\infty}$ behaves like the usual scattered fields and far-field patterns in the acoustic scattering. Further, since the proofs of Theorem {\rm \ref{thm6.7} and \ref{thm6.8}} rely heavily on the reciprocity relations, it can be verified that similar results also hold for $\mathcal{B}(u)=(\Delta u,\pa_n\Delta u)$. 
\end{remark}

\section*{Acknowledgements}
This work was supported by the NNSF of China with grant 12122114.


\begin{thebibliography}{99}
\bibitem{YA16} Y. Assylbekov, Inverse problems for the perturbed polyharmonic operator with coefficients in Sobolev spaces with non-positive order, {\em Inverse Probl. \bf32}(2016), 105009, 22pp. 
	
\bibitem{LC20} L. Bourgeois and C. Hazard, On well-posedness of scattering problems in a Kirchhoff-Love infinite plate, {\em SIAM J. Appl. Math. \bf80}(2020), 1546-1566. 
	
\bibitem{LA20} L. Bourgeois and A. Recoquiliay, The linear sampling method for Kirchhoff-Love infinite plates, {\em Inverse Probl. Imaging \bf14}(2020), 363-384. 
	
\bibitem{FD06} F. Cakoni and D. Colton, {\em Qualitative Methods in Inverse Scattering Theory}, Springer, Berlin, 2006.
	
\bibitem{AH17} A. Choudhury and H. Heck, Stability of the inverse boundary value problem for the biharmonic operator: logarithmic estimates, {\em J. Inverse Ill-posed Probl. \bf25}(2017), 251-263. 
	
\bibitem{AH13} A. Cossonni\`{e}re and H. Haddar, Surface integral formulation of the interior transmission problem, {\em J. Integral Equ. Appl. \bf25}(2013), 341-376. 
	
\bibitem{DR13} D. Colton and R. Kress, {\em Inverse Acoustic and Electromagnetic Scattering Theory} (3rd Ed.),
Springer, New York, 2013.
	
\bibitem{HP23} H. Dong and P. Li, A novel boundary integral formulation for the biharmonic wave scattering problem, {\em J. Sci. Comput. \bf98}(2024), 42. 
	
\bibitem{FHG10} F. Gazzola, H.-C. Grunau and G. Sweers, {\em Polyharmonic Boundary Value Problems}, Springer-Verlag, Berlin, 2010. 

\bibitem{MJV22} M. Harju, J. Kultima and V. Serov, Inverse scattering for three-dimensional quasi-linear biharmonic operator, {\em J. Inverse Ill-posed Probl. \bf30}(2022), 379-393. 

\bibitem{GW08} G. Hsiao and W. Wendland, {\em Boundary integral equations}, Springer, New York, 2008. 

\bibitem{AK89} A. Kirsch, Surface gradients and continuity properties for some integral operators in classical scattering theory, {\em Math. Methods Appl. Sci. \bf11}(1989), 789-804. 

\bibitem{AR93} A. Kirsch and R. Kress, Uniqueness in inverse obstacle scattering, {\em Inverse Probl. \bf9}(1993), 285-299.

\bibitem{KMG14} K. Krupchyk, M. Lassas and G. Uhlmann, Inverse boundary value problems for the perturbed polyharmonic operator, {\em Trans. Amer. Math. Soc. \bf366}(2014), 95-112. 

\bibitem{NL72} N. N. Lebedev, {\em Special Functions and Their Applications}, Dover Publications, Inc., New York, 1972. 

\bibitem{MKG12} M. Lassas, K. Krupchyk and G. Uhlmann, Determining a first order perturbation of the biharmonic operator by partial boundary measurements, {\em J. Funct. Anal. \bf262}(2012), 1781-1801. 

\bibitem{PX22} P. Li and X. Wang, An inverse random source problem for the biharmonic wave equation, {\em SIAM /ASA J. Uncertain. Quantif.\bf10}(2022), 949-974. 

\bibitem{PX21} P. Li, X.Yao and Y. Zhao, Stability for an inverse source problem for the damped biharmonic plate equation, {\em Inverse Probl. \bf37}(2021), 085003, 19pp. 

\bibitem{PXY21} P. Li, X.Yao and Y. Zhao, Stability for an inverse source problem of the biharmonic operator, {\em SIAM J. Appl. Math. \bf81}(2021), 2503-2525. 

\bibitem{NR09} N. V. Movchan, R. C. McPhedran, A. B. Movchan and C. G. Poulton, Wave scattering by platonic grating stacks, {\em Proc. R. Soc. A. \bf465}(2009), 3383-3400. 

\bibitem{WFR66} W. Magnus, F. Oberhettinger and R. P. Soni, {\em Formulas and Theorems for the Special Functions of Mathematical Physics}, Springer-Verlag, New York, 1966. 

\bibitem{JC01} J.-C. Nedelec, {\em Acoustic and Electromagnetic Equations, Integral Representations for Harmonic Problems}, Springer-Verlag, New York, 2001. 

\bibitem{TV18} T. Tyni and V. Serov, Scattering problems for perturbations of the multidimensional biharmonic operator, {\em Inverse Probl. Imaging \bf12}(2018), 205-227. 

\bibitem{GW94} G. Watson, {\em A Treatise on the Theory of Bessel Functions}, Cambridge University Press, Cambridge, England; The Macmillan Company, New York, 1944. 

\bibitem{YY14} Y. Yang, Determining the first order perturbation of a bi-harmonic operator on bounded and unbounded domains from partial data, {\em J. Differ. Equations \bf257}(2014), 3607-3639. 








\end{thebibliography}
\end{document}